\documentclass[preprint, 12pt]{elsarticle}

\usepackage{theorem}
\usepackage{latexsym}
\usepackage{amsfonts}
\usepackage{amssymb}
\usepackage{amsmath}
\usepackage{graphicx}

\usepackage{pstricks}

\newtheorem{theorem}{Theorem}[section]
\newtheorem{lemma}[theorem]{Lemma}
\newtheorem{corollary}[theorem]{Corollary}
\newtheorem{example}[theorem]{Example}
{\theorembodyfont{\rmfamily}\newtheorem{definition}[theorem]{Definition}}
{\theorembodyfont{\rmfamily}\newtheorem{remark}[theorem]{Remark}}

\newproof{proof}{Proof}
\newcommand{\beq}{\begin{eqnarray*}}
\newcommand{\eeq}{\end{eqnarray*}}

\def\R{{\mathbb{R}}}
\def\N{{\mathbb{N}}}
\def\Z{{\mathbb{Z}}}

\begin{document}
\begin{frontmatter}
\title{Norm Hilbert spaces over $G$-modules with a convex base\footnote{Partially supported by DIUFRO DI150043}}
\author{E. Olivos - H. Ochsenius}
\ead{elena.olivos@ufrontera.cl, herminia.ochsenius.a@gmail.com,}
\address{Universidad de La Frontera. Temuco. Chile}

\begin{abstract} By analogy with the classical definition, a Norm Hilbert space $E$ is defined as a Banach space over a valued field $K$ in which each closed subspace has an orthocomplement. In the rank one case (that is, the value group as well as the set of norms of the space are contained in $[0, \infty)$), they were described by van Rooij in his classical book of 1978, but the name itself was introduced in 1999 by Ochsenius and Schikhof  for the case of spaces with an infinite rank valuation.

Here we shall also consider only value groups that are contained in $(\R^+,\cdot)$, yet we borrow from the infinite rank case the notion of a $G$-module for the set of norms of the space. Their  structure allows for greater complexity than that of ordered subsets of $\R$. In this paper we describe a new class of Norm Hilbert spaces,  those in which the $G$-module has a convex base. Their characteristics will be the focus of our study.

\end{abstract}


\end{frontmatter}

\maketitle

\section*{Introduction}

The relevant place of real and complex Hilbert spaces in so many areas of mathematics has made the question of its generalization to spaces over different scalar fields a most interesting one. The decision of which features of the classical definition should be maintained, and which will be allowed to change, has been central. A well developed theory has as its center normed spaces over valued fields, and crucial questions are related to areas of Functional Analysis.  A comprehensive study can be found in van Rooij's  classical book (see \cite{Rooij}). A different approach started with the work of Keller, Gross and K\"unzi. They were interested in spaces with a bilinear form that induced a norm, and in which the Projection Theorem was valid. They termed them orthomodular spaces (see \cite{G-Ku} and \cite{Keller}).\\
In both cases fields have a non-archimedean valuation and the norm is also non-archimedean. But in the first one the value group is a subset of $\R$ whereas in the second it is the union of the infinite  set of its convex subgroups. They are said to be valuations of rank one or of infinite rank respectively.\\
As the theory of generalized Hilbert spaces (that is, every closed  subspace has an orthocomplement) over arbitrary fields developed, it became clear  that in many interesting cases the set of vector norms was not related directly to the valued field or to the value group. Thus a far more general concept was introduced as a ``home'' for the set of norms. It was called a $G$-module, and the general setup was structured in \cite{morado} and further developed in \cite{base}.\\
A new question was posed by Schikhof and Olivos: How do the features of the $G$-module intertwine with those of the vector spaces? To be precise, suppose  that the value group $G$ is just a cyclic group, but place no conditions a priori  on the $G$-module except that it has a convex base as in \cite{base}. What can be said of NHS spaces?
This is the central theme of this paper. Finally we could like to recognize the debt that this paper has towards the research of W. Schikhof. He started the study of $G$-modules and the way their specific characteristics influenced the properties of non-archimedean vector spaces. The present work follows a line originated by him and the authors during his stays in Temuco, Chile.\\

We start with the definitions that shall be needed in what follows, the prominent one being that of a $G$-module with a convex base. Then we study the relevant theorems, first  for Banach spaces and then for NHS, in this new setup. We give full proofs, even in the cases when they are quite similar to classical ones. The paper ends with a new characterization of NHS that is both simple and useful.

\section{Basic concepts and notation}

\begin{definition} Let $\langle G,\, \cdot\rangle $ be an ordered group.
A linearly ordered set $(X, \leq)$ with no smallest element is called a {\bf $G$-module} if $G$ acts on $X$ by the map $(g,\, x)\mapsto gx$ and for any $g,\, g_1,\, g_2\in G,\ x,\, x_1,\, x_2\in X$ we have\\
(i) $g_1\leq g_2 \Rightarrow g_1x\leq g_2x$,\\
(ii) $x_1\leq x_2 \Rightarrow gx_1\leq gx_2$,\\
(iii) The orbit $Gx$ is coinitial in $X$
\end{definition}

\begin{remark}
It is easily proven that $Gx$ is also cofinal in $X$ and that $X$ has no largest element.\\
Usually we adjoin a smallest element to $X$, that will be denoted by 0.\end{remark}

\begin{definition} Let $X$ and $Y$ be two $G$-modules, a map $\varphi: X\to Y$ is called a $G$-module map if $\varphi$ is increasing and $\varphi(gx)=g\varphi(x)$ for all $g\in G,\  x\in X$.\end{definition}

We introduce now a central concept, the {\bf convex base} of a $G$-module $X$. \\
A subset $B$ of $X$ is called a {\bf generating set} if $GB=X$, {\bf independent} if for $b_1\neq b_2\in B$ we have $Gb_1\cap Gb_2=\varnothing$, and a {\bf base} of $X$ if $B$ is a set that is both generating and independent. Clearly  $B$ is a base of $X$ if and only if $B$ contains one and only one element of any orbit $Gx$. Therefore two bases of a $G$-module $X$ have the same cardinality.

The next definition is not surprising,
\begin{definition} A convex base of a $G$-module $X$ is a base $B$ of $X$ that is convex as a subset of $X$.\end{definition}

\begin{remark} $G$-modules with a convex base were studied in \cite{base} for $G$ any abelian, multiplicatively written, totally ordered group. A strong characterization of them was given in Theorem 2.9. Here we restate this result for our case.\end{remark}

\begin{theorem} Let $G$ be a cyclic group and $X$ a $G$-module. The following are equivalent.\\
(i) $X$ has a convex base.\\
(ii) As a $G$-module, $X$ is isomorphic to $B\times G$ for some non empty chain $B$ (where the action of $G$ on $B\times G$ is defined by $g'(b,\, g)=(b,\ g'g)$).\end{theorem}

\begin{remark} By this theorem, given any cyclic group $G$ and any cardinal $\alpha>\aleph_0$ we can obtain a $G$-module $X$ with convex base, such that the cardinality of $X$ is $\alpha$.\end{remark}

We give now two examples of $G$-modules with a convex base.\\
  
Let $(K,\, |\ |)$ be a valued field, with value group $G$, a cyclic subgroup of $(\R^+, \cdot)$.

\begin{example}\label{ex1} Define $X_1$ as $B_1\times G$, with $B_1:=(0,\, 1]\subset\R^+$ (Notice that $B_1$ is not a well ordered set, since it is not isomorphic to an ordinal).\end{example}
\begin{example}\label{ex2} Now let $X_2 := B_2\times G$ where $B_2$ is the ordinal $\omega_1$. (Notice that $X_2$ cannot be immersed in $\R^+$, since $\omega_1$ has no cofinal sequence).\end{example}

\section{Banach spaces over a discretely valued field with norms on a $G$-module with convex base}

We shall study NHS in the case the field $K$ has a valuation of rank one.\\

To prove the following lemma we will use the ``main tool'' of \cite{two}. \\
Let $E=(E,\,\|\ \|)$ be an $X$-normed space and $x_0\in X$. The $G$-module map $\varphi:X\to G^\#$ defined by $\varphi(x)=\sup_{G^\#}\{g\in G: gx_0\leq x\}$ induces a new norm in $E$ with values in $G^\#$. It turned out that these two norms are equivalent. The space $E$ provided with the new norm is called $E_\varphi$.

\begin{lemma}\label{discrete} If $E$ is a NHS over a field $K$ whose valuation has rank one then the value group of $K$ is cyclic.\end{lemma}

\begin{proof}
If the value group $G\subseteq (0, \infty)$ then so is its Dedekind completion.\\
By \cite{two}  Theorem 2.7 if $E$ is a NHS then so is $E_\varphi$ and by \cite{Rooij} Theorem 5.16, if each closed subspace of $E_\varphi$ has an orthocomplement then the valuation is discrete.\end{proof}

This has a strong consequence.

\begin{theorem}\label{Ban1.1} Let $G=\langle g_0\rangle$ be a cyclic group. Each $G$-module $X$ has a convex base. In fact, $A:=[a,\, g_0a)$ $($as well as $(a,\, g_0a] )$ is a convex base of $X$  for any $a\in X$.\end{theorem}
\begin{proof} It is proved in \cite{base}, Lemma 4.10  that for any $a\in X$ the set $A$ is a convex base of the submodule $GA$. Thus, we only need to show that $G[a,\, g_0a)=X$. So, let $x\in X$. As $\{g_0^nx\}_{n\in\Z}$ is cofinal and coinitial there exists $m\in\Z$ such that $g_0^mx<a\leq g_0^{m+1}x$. Then $g_0^{m+1}x<g_0a$ so $g_0^{m+1}x\in [a,\, g_0a)$ i.e. $x\in G[a,\, g_0a)$.

\end{proof}

Thus in the spaces described by van Rooij (\cite{Rooij}, theorems 5.13 and 5.16) the set of norms has, as a $G$-module, always a convex base.\\

But the standard construction given below shows that there are a host of possibilities for new examples of spaces in which the set of norms is a $G$-module with a convex base.

\begin{example}
Let $B$ a chain, $E$ be the space of linear combinations of the set $V:=\{v_b: b\in B\}$ over a valued field $K$ with value group $G=\langle g_0\rangle$, a cyclic group. We define for each $b\in B$ the norm $\|v_b\|=(b, 1)$ and for any $v:=\sum\limits_{b}k_b v_b \in E$, $\|v\|=\max\{\|k_bv_b\|\}$. Therefore $\|E\setminus\{0\}\|=B\times G$.
\end{example}

From now on $G=\langle g_0\rangle$ is a cyclic group, $X$ is a $G$-module and  $E$ is an $X$-normed Banach space (that is, complete in the topology of the $X$-valued norm function).\\
Notice that by Theorem \ref{Ban1.1}, the $G$-module $X$ has a convex base.

\begin{lemma}\label{Ban2.1} Let $E$ be an $X$-normed Banach space over $K$. Then the valuation topology on $K$ and the norm topology on $E$ are (ultra)metrizable.\end{lemma}
\begin{proof} Fixing an $a\in E$, $a\neq 0$ we see that $\{\lambda\in K: |\lambda|<g_0^n\}_{n\in\Z}$ and $\{ v\in E:\|v\|<g_0^n\|a\|\}_{n\in\Z}$ form neighbourhood bases of 0 in $K$ and $E$, respectively.\end{proof}

\begin{theorem}\label{Ban2.11} Let $E$ have an orthogonal base $\{e_1,\, e_2,\, \ldots\}$. Then there is an orthogonal base $\{f_1,\, f_2,\, \ldots\}$ such that $\|f_1\|>\|f_2\|>\cdots$ and $f_n\to 0$.\end{theorem}

\begin{proof} Let $s_1,\, s_2,\, \ldots \in X$ such that $s_1>s_2>\cdots$ and $s_n\to 0$. Put $f_1:=e_1$. There is an $n_1\in\Z$ such that $\|g_0^{n_1}e_2\|<\min\{\|e_1\|,\, s_1\}$. Put $f_2:=g_0^{n_1}e_2$, and so on. Inductively we arrive at $f_1,\, f_2,\, \ldots$, multiples of $e_1,\, e_2,\,\ldots$ respectively (hence orthogonal) such that $\|f_n\|<s_n$ for each $n\in\N$. Hence $f_n\to 0$.\end{proof}

We now prove the following  theorems on the structure of Banach spaces. It appears in a more general setup in \cite{morado}, but we include it with its short proof for convenience.

\begin{theorem}\label{Ban2.5} A Banach space of countable type has an orthogonal base.\end{theorem}

\begin{proof} We may assume $E$ is an infinite dimensional Banach space of countable type. Then there exists a linearly independent set $\{v_1,\, v_2,\,\ldots\}$ whose linear hull is dense. Using spherical completeness of $D_n:=[v_1,\, v_2,\,\ldots,\, v_n]$ we construct inductively $e_1,\, e_2,\,\ldots\in E$ such that $[e_1,\, e_2,\,\ldots,\, e_n]=D_n$ and $e_{n+1}\perp D_n$ for each $n$. Then clearly $e_1,\, e_2,\,\ldots$ is an orthogonal base for $E_n$.
\end{proof}
\bigskip
{\bf We consider now $c_0(I)$ and define the expression ``to contain $c_0$''.}\\

The wording of the following definitions, given in \cite{Rooij}, have been slightly modified in order to include our present cases.

\begin{definition} Let $I$ be an index set, the space $c_0(I)$ is the space of all $(\lambda_i)_{i\in I}\in \prod\limits_IK$ for which
the set $ \{i\in I: |\lambda_i|\geq g\}$  is finite for all $g\in G$.\\
It is a Banach space with respect to the $G$-norm $(\lambda_i)_{i\in I}\mapsto \max_i|\lambda_i|$. The unit vectors of $c_0(I)$ form an orthogonal base of $c_0(I)$. As usual we will write $c_0$ for $c_0(\N)$.
\end{definition}
We have

\begin{definition} A $K$-Banach space $E$ is said  {\bf to contain $c_0$ } if there exist $a\in X$ and an orthogonal sequence $e_1,\, e_2,\,\ldots$ in $E$ such that $\|e_n\|=a$ for each $n$.
\end{definition}

\begin{theorem}\label{Ban2.14} Let $E$ be a Banach space, let $X:=\|E\setminus\{0\}\|$ with convex base $B$. The following are equivalent.

($\alpha$) $E$ does not contain $c_0$.

($\beta$) For all orthogonal systems $\{e_i:i\in I\}$ we have that the set $\{i\in I: \|e_i\|\in Gb\}$ is finite for each $b\in B$.

($\gamma$) There is a maximal orthogonal system  $\{e_i:i\in I\}$ such that the set $\{i\in I:\|e_i\|\in Gb\} $ is finite for each $b\in B$.\end{theorem}

\begin{proof} 
($\alpha\Rightarrow \beta$) We may assume $I$ infinite. If $\{i\in I: \|e_i\|\in Gb\}$ is infinite for some $b\in B$ we can select $\lambda_i\in K^*$, ($i\in I$) and $i_1,\, i_2,\, \ldots\in I$ (distinct) such that $\|\lambda_{i_n}e_{i_n}\|=b$ for each $n$, conflicting with ($\alpha$).

($ \beta\Rightarrow\gamma$) is obvious. By Zorn the existence of a maximal orthogonal system is assured.

($\gamma \Rightarrow\alpha$) If $\{f_i:i\in I\}$ is a second maximal orthogonal system then by \cite{morado} 2.4.12, $\{i\in I: \|f_i\|\in Gb\}$ is finite for each $b\in B$. Now if $E$ has orthogonal vectors $a_1,\, a_2,\, \dots$ of the same length we can extend $\{ a_1,\, a_2,\, \ldots\}$ to a maximal orthogonal system. This one violates ($\gamma$).

\end{proof}

\section{Norm Hilbert spaces over $G$-modules with a convex base}

\begin{definition} An $X$-normed Banach space is a Norm Hilbert space (NHS) iff every closed subspace $D$ of $E$ is orthocomplemented. \end{definition}
This section is the center of our research. We start with several characterization of Norm Hilbert spaces very much in the spirit of van Rooij's classical book  ( \cite{Rooij} Theorems 5.13 and 5.16). But in his work the set of norms is a subset of $(0, \infty)$ and in our case some assertions do not necessarily hold. Thus we prefer to give detailed proofs of the equivalences in theorem \ref{Ban2.3}.

\begin{remark}\label{NHS3.1} Recall (\cite{morado} Proposition 4.1.2) that $E$ is a NHS iff each maximal orthogonal system in $E$ is an orthogonal base, iff for each closed subspace $D$ the set $\{\|v-d\|: d\in D\}$ has a minimum in $X\, \cup\{0\}$.  \end{remark}

\begin{theorem}\label{Ban2.3} Let $K$ be a valued field, $G$ a cyclic group, value group of $K$, $X$ a $G$-module with convex base and $E$ a Banach space over $K$. Then the following are equivalent.

\begin{itemize}
\item[($\alpha$)] $E$ is a NHS.
\item[($\beta$)] If $e_1,\, e_2,\, \ldots\in E$ are orthogonal and $\|e_1\|>\|e_2\|>\cdots$ then $e_n\rightarrow 0$.
\item[($\gamma$)] If $v_1,\, v_2,\, \ldots\in E$ and $\|v_1\|> \|v_2\|>\cdots$ then $v_n\rightarrow 0$.
\item[($\delta$)] Each closed subspace of countable type is a NHS.
\item[($\epsilon$)] Each closed hyperplane is orthocomplemented.
\item[($\iota$)] $E$ has an orthogonal base and is spherically complete.
\end{itemize} \end{theorem}

\begin{proof}

($\gamma$)$\Rightarrow$($\alpha$):  It suffices to prove (\ref{NHS3.1}) that each maximal orthogonal system $\{e_i:i\in I\}$ is an orthogonal base. So let $D:=\overline{[e_i: i\in I]}$; we must prove $D=E$. Suppose $v\in E\setminus D$, we derive a contradiction. Consider $V:=\{\|v-d\|:d\in D\}$. If $\|v-d_0\|=\min V$ then $v-d_0\perp D$, a contradiction. So $\min V$ does not exist and we can therefore find $d_1,\, d_2,\,\ldots\in D$ such that $\|v-d_1\|>\|v-d_2\|>\cdots$. By assumption $\|v-d_n\|\rightarrow 0$. But then $v\in\overline{D}=D$ and we have our contradiction.

($\alpha$)$\Rightarrow$($\beta$): Suppose $e_1,\, e_2,\, \ldots\in E$ orthogonal, $\|e_1\|>\|e_2\|>\cdots >s$ for some $s\in X$. We derive a contradiction. Consider $D:=\overline{[e_i: i\in I]}$ and $\phi\in D'$ given by $\phi\left(\sum \xi_ie_i\right)=\sum \xi_i $, ($\xi_i\rightarrow 0$).\\
Now $D$ is a NHS, so $\ker\phi$ has an orthocomplement $Ka$ in $D$. Without lost of generality $\phi(a)=1$. Let $a:=\sum\limits_{n_1}^\infty\lambda_ne_n$. We have $$1=|\phi(a)|=\left|\sum\limits_{n=1}^\infty\lambda_n\right|\leq\max\limits_n|\lambda_n|$$ We see that there exists $i\in\N$ with $|\lambda_i|\geq 1$, $\|a\|=\max\|\lambda_ne_n\|\geq \|\lambda_ie_i\|\geq \|e_i\|>\|e_{i+1}\|$.\\
But $\phi(a-e_{i+1})=0$, so $a-e_{i+1}\in \ker\phi$, so $a\perp a-e_{i+1}$, and so $|e_{i+1}|=\max\{ \|a\|,\,\|a-e_{i+1}\|\}\geq a$. Contradiction.

To complete the link we shall prove, by contradiction, ($\beta$)$\Rightarrow$($\gamma$): Suppose we have a decreasing sequence  $\|v_1\|> \|v_2\|>\cdots>s$ ($v_1,\, v_2,\, \ldots \in E,\, s\in X$). \\
Let $B$ be a convex base of $X$, thus $X=\bigcup g_0^nB$. For any fixed $m\in\Z$ we observe that $I_m:=\{ n\in\Z: \|v_n\|\in g_0^mB\}$ is finite (If $I_m$ is infinite then $\{\|v_n\|:n\in I_m\}$ would consist of elements that are not equivalent mod $G$.Thus they would be orthogonal, but this is forbidden by assumption). But now, $s\in g_0^mB$ for some $m$. We see that $\|v_i\|>g_0^{m-1}B$ for all $i$. Also $v_1\in g_0^rB$ for some $r$. Hence the whole sequence $v_1,\, v_2,\, \ldots$ is contained in $g_0^rB,\, g_0^{r+1}B,\, \ldots,\, g_0^{m-1}B$. This implies finiteness of $v_1,\, v_2,\, \ldots$, a contradiction.

This completes the proof of the equivalence.

Clearly we have ($\alpha$)$\Rightarrow$($\delta$). Now ($\delta$)$\Rightarrow$($\gamma$) is easy by observing that $\overline{[v_1,\, v_2,\,\ldots]}$ is of countable type. ($\alpha$)$\Rightarrow$($\epsilon$) is trivial. 

To prove ($\epsilon$)$\Rightarrow$($\alpha$), let $(e_i)_{i\in I}$ be a maximal orthogonal system. It suffices to prove (\ref{NHS3.1}) that $D:= \overline{[e_i: i\in I]}=E$.  Suppose not. Then take an $a\in E\setminus D$ and consider the map $\lambda a+d\mapsto \lambda$, ($\lambda\in K, \, d\in D$) which is in $(Ka+D)'$. By the Hahn- Banach Theorem (\cite{morado}), $f$ extends to a $g\in E'$. Then $H:= {\rm Ker}\, g$ is a closed hyperplane as $a\in H$. By assumption there is a $z\in E\setminus H$ with $z\perp H$. But then, since $D\subseteq H$, also $z\perp D$ which conflicts with the maximality of $\{e_i: i\in I\}$.

We now prove  ($\alpha$)$\Rightarrow$($\iota$).  Clearly $E$ has an orthogonal base. To prove spherical completeness  let $\{B_i\}_{i\in I}$ be a nest of balls in $E$ where $I$ is linearly ordered and $i<j \Rightarrow B_i\supseteq B_j$. By the remark following Definition 1.2.1 in \cite{morado}, we may suppose that $B_i$ has the form $\{v\in E:\|v-a_i\|\leq r_i\}$, where we may suppose that $r_i<r_j$ if $i>j$.

To prove $\bigcap B_i\neq \varnothing$ we may suppose that $I$ has no smallest element. Then there are $i_1,\, i_2,\,\ldots\in I$ which $r_{i_1}> r_{i_2}>\cdots$. \\
By ($\gamma$) we must have $r_{i_n}\rightarrow 0$ showing that $\bigcap B_i=\bigcap\limits_{n\in\N}B_{i_n}$ is a singleton set by ordinary completeness of $E$.

Finally, we prove ($\iota$)$\Rightarrow$($\alpha$): Let $\{f_i\}_{i\in I}$ be a maximal orthogonal set in $E$. We prove that $D:=\overline{[f_i: i\in I]}=E$. \\
Now $E$ has an orthogonal base (which has the cardinality of $I$), say $\{e_i\}_{i\in I}$.\\
As $\{e_i\}_{i\in I}$ is also a maximal orthogonal system we have that $D$ is isometrically isomorphic to $\overline{[e_i: i\in I]}=E$. Thus $D$ is spherically complete, so for each $v\in E\setminus D$, $\min\{\|v-d\|: d\in D\}$ exists, so that $v-d\perp D$. This conflicts the maximality, so we have $D=E$.

This completes the proof of Theorem \ref{Ban2.3}.
\end{proof}

\begin{corollary}\label{Ban2.4} For each index set $I$ the space $c_0(I)$ is a NHS.\end{corollary}

\begin{proof} We have $\|c_0(I)\|=G\,\cup\, \{0\}$. Clearly every strictly decreasing sequence on $G$ tends to 0, so $c_0(I)$ satisfies ($\gamma$) of \ref{Ban2.3}, and we are done.
\end{proof}

We come back now to $c_0(I)$ and its link with isometries.

\begin{remark} As it has been said before, we find in \cite{morado} Lemma 4.3.4, that a NHS over a field with an infinite rank valuation does not contain $c_0$.\end{remark}

\begin{definition} $E$ is {\bf rigid} if every linear isometry $ T: E\to E$ is surjective.\end{definition}
The subject was studied in \cite{linear} for the case of NHS over fields with a valuation of (countable) infinite rank. Let $E$ be such a space, then $E$ is rigid, in fact no proper subspace can be isometrically isomorphic to the whole space and any isometry of a closed subspace into itself can be extended to an isometry of $E$ onto itself. Clearly this sharply contrast the case of the classical space $c_0$.

The following theorems show that the condition ``$E$ does not contain $c_0$'' will need additional hypotheses in order to ensure that $E$ is rigid.

\begin{theorem}\label{Ban2.7} If $E$ is a spherically complete space that does not contain $c_0$ then $E$ is rigid.\end{theorem}
\begin{proof} Let $T:E\to E$ be a linear isometry, $TE\neq E$; we derive a contradiction. Now $TE$ is spherically complete, so it has an orthocomplement in $E$, in particular, there exists a nonzero $v\in E$, $v\perp TE$. Inductively we find easily that $\{v,\, Tv,\, T^2v,\ldots \} $ is orthogonal and that $\|T^nv\|=\|v\|$ for each $n$, so $E$ contains $c_0$, a contradiction.

\end{proof}
We obtain
\begin{corollary}\label{Ban2.8} $E$ is a NHS and does not contain $c_0$  if and only if  $E$ is rigid  and has an orthogonal base.\end{corollary}

Not surprisingly, rigidity is a condition that forbids $E$ to contain $c_0$.

\begin{theorem}\label{Ban2.9} A rigid space does not contain $c_0$.\end{theorem}
\begin{proof} Suppose $E$ is rigid and contains $c_0$, that is there are $a_1,\, a_2,\, \ldots$, orthogonal, $\|a_n\|=s$ for all $n$. Then $\overline{[a_1,\, a_2,\, \ldots]}$ is a NHS, it is spherically complete so it has an orthocomplement $D$.\\
Then define $T$ in  $\overline{[a_1,\, a_2,\, \ldots]}$ by $Ta_n=a_{n+1}$ and $T(d)=d$ for all $d\in D$. This gives us a nonsurjective linear isometry, a contradiction.
\end{proof}

\begin{theorem}\label{Ban2.10} Let $E$ have an orthogonal base $\{e_1,\, e_2,\, \ldots\}$ for which $n\neq m$ implies $\|e_n\|\notin G\|e_m\|$. Then $E$ does not contain $c_0$.\end{theorem}
\begin{proof} Suppose we have an orthogonal set $\{a_1,\, a_2,\,\ldots\}$ and $s\in X$ such that $\|a_n\|=s$ for all $n$; we derive a contradiction. For each $n\in\N$ we have an expansion $$a_n=\sum_{i=1}^\infty \lambda_j^ne_j\hspace{1cm} (\lambda_j^n\in K)$$
There is a unique $j_0$ with $s\in G\|e_{j_0}\|$. Thus $$s=\|a_n\|=\max\limits_j\|\lambda_j^ne_j\|=\|\lambda_{j_0}^ne_{j_0}\|$$ Clearly we have for each $n$, $\|a_n-\lambda_{j_0}^ne_{j_0}\|<\|a_n\|$. Therefore, by the Perturbation Lemma, the sequence $n\mapsto \lambda_{j_0}^ne_{j_0}$ must be orthogonal, an absurdity.\end{proof}

\begin{theorem}\label{Ban2.16} Let $E$ have an orthogonal base and suppose there exists a sequence  $\{v_n\}$ in $E$ with $\|v_1\|>\|v_2\|>\cdots$, $\|v_n\|\nrightarrow 0$. Then $E$ is not rigid.
\end{theorem}
\begin{proof} By  Theorem \ref{Ban2.3} $E$ is not a NHS, so by Corollary \ref{Ban2.8} $E$ cannot be rigid.

\end{proof}

Now it is easy to present an example of a space that does not contain $c_0$ but is not rigid. Let $B=\{ b_1,\, b_2,\,\ldots\}\subseteq\R^+$ a denumerable chain where $b_1>b_2>\cdots$. Let $X:= B\times G$ ($G$ a cyclic group) thus $X$ has a convex base $B\times\{1_G\}$. Let $E$ have an orthogonal base $\{ e_1,\, e_2,\, \ldots\} $ with $\|e_n\|=(b_n, \, 1_G)$ for each $n$. Clearly by \ref{Ban2.10} $E$ does not contain $c_0$. By \ref{Ban2.16} $E$ is not rigid.\\

The next section contains the surprising main result of this paper.

\section{A new characterization of NHS}
We recall the standard definition: A linear ordering $\leq$ of a set $S$ is a well-ordering if every non-empty subset of $S$ has a least element.
\begin{theorem}\label{Ban2.12} Let $K$ be a valued field and $G=\langle g_0\rangle$ its cyclic value group. $E$ is a $K$-Banach space and $X:=\|E\setminus\{0\}\|$, the set of norms, is a $G$-module with convex base $B$.\\
Then $E$ is a NHS if and only if $B$ is well ordered.\end{theorem}

\begin{proof} Let $E$ be a NHS, let $b_1>b_2>\cdots$ where $b_n\in B$; we derive a contradiction. There are $v_1,\, v_2,\,\ldots\in E$ with $\|v_n\|=b_n$ for each $n$. Then $\|v_1\|>\|v_2\|>\cdots$. But by \ref{Ban2.3}, $\|v_n\|\to 0$ an impossibility as for all $n$, $b_n>g_0^{-1}B$ .

Conversely. let $B$ be well-ordered, let $v_1,\, v_2,\, \ldots \in E$. $\|v_1\|>\|v_2\|>\cdots$. We have $X=\bigcup\limits_{n\in\Z}g_0^nB$, where $$\cdots< g_0^{-1}B<B<g_0B<g_0^2B<\cdots$$ Since $\{n\in \N:\|v_n\|\in g_0^mB\}$ is finite for each $m\in \Z$ (as $g_0^mB$ is well-ordered) we must have, for each $m\in\Z$, that $\|v_n\|<g_0^mB$ eventually, i.e. $\|v_n\|\to 0$ so $E$ is a NHS by Theorem \ref{Ban2.3}.\end{proof}

\begin{remark} The previous theorem is a strong characterization of NHS over $G$-modules $X$ with a convex base and $G$ a cyclic group. If we come back to examples \ref{ex1} and \ref{ex2} it is now clear that an $X_1$-Banach space can never be a NHS, while an $X_2$-Banach space is necessarily a NHS.\\

In addition, as the value group $G=\langle g_0\rangle$ is a cyclic group, we have by Theorem 2.1 that  for any $a \in X$ the interval $[a,g_0a)$ is a convex base of $X$. With this in mind we show in the next Corollary how easy it is to decide if the space $E$ is a NHS.

 \end{remark}

\begin{corollary} Let $K$, $G$, $X$, $E$ be as in Theorem \ref{Ban2.12}. Let $a\in X$, then $E$ is a NHS if and only if the interval $[a,\, g_0a)$ is well ordered.\end{corollary}

\end{document}